\documentclass[12pt,a4paper]{article}
\usepackage[utf8]{inputenc}
\usepackage[T1]{fontenc}
\usepackage{authblk}

\usepackage{amsmath, amscd, amsthm, amscd, amsfonts, amssymb, graphicx, color, soul, enumerate, xcolor,  mathrsfs, latexsym, bigstrut, framed, caption}
\usepackage{pstricks,pst-node,pst-tree}
\usepackage[bookmarksnumbered, colorlinks, plainpages,backref]{hyperref}
\hypersetup{colorlinks=true,linkcolor=blue, anchorcolor=green, citecolor=midblue, urlcolor=darkblue, filecolor=magenta, pdftoolbar=true}
\usepackage[all]{xy}

\usepackage{tikz}
\usetikzlibrary{calc,decorations.pathreplacing,decorations.markings,positioning,shapes}

\newcommand{\Aut}{\ensuremath{\operatorname{Aut}}}

\newcommand{\id}{\ensuremath{\text{\rm id}}}

\definecolor{cupgreen}{rgb}{0,0.498,0.208}
\definecolor{cupblue}{rgb}{0,0,.5}
\definecolor{cupred}{rgb}{1,0.04,0}
\definecolor{cuppink}{rgb}{0.925,0,0.545}
\definecolor{cupmagenta}{rgb}{0.624,0.161,0.424}
\definecolor{cupbrown}{rgb}{0.71,0.212,0.133}

\definecolor{cupgreen}{rgb}{0,0,0}
\definecolor{cupblue}{rgb}{0,0,0}
\definecolor{cupred}{rgb}{0,0,0}
\definecolor{cuppink}{rgb}{0,0,0}
\definecolor{cupmagenta}{rgb}{0,0,0}
\definecolor{cupbrown}{rgb}{0,0,0}

\definecolor{TITLE}{rgb}{0,0,0}
\definecolor{midblue}{rgb}{0.00,0.0,0.80}
\definecolor{darkblue}{rgb}{0.00,0.00,0.45}
\definecolor{SECTION}{rgb}{0.50,0.00,1.00}
\definecolor{THM}{rgb}{0.8,0,0.1}
\definecolor{SEC}{rgb}{0,0,1}

\newcommand{\aut}{\mathrm{Aut}}

\makeatother
\normalsize
\newtheorem{theorem}{{\color{THM} Theorem}}[section]

\DeclareRobustCommand{\stirling}{\genfrac\{\}{0pt}{}}

\newtheorem{lemma}[theorem]{{\color{THM}Lemma}}

\newtheorem{corollary}[theorem]{{\color{THM}Corollary}}

\theoremstyle{definition}

\numberwithin{equation}{section}

\date{}

\title{Number of colors needed to break symmetries of a graph by an arbitrary edge coloring}

\author{Saeid Alikhani\thanks{alikhani@yazd.ac.ir}}

\author{Mohammad H. Shekarriz\thanks{mhshekarriz@yazd.ac.ir}}

\affil{Department of Mathematics, Yazd University, 89195-741, Yazd, Iran.}

\begin{document}
	
	\maketitle
	\begin{abstract}
		A coloring is distinguishing (or symmetry breaking) if no non-identity automorphism preserves it. The distinguishing threshold of a graph $G$, denoted by $\theta(G)$, is the minimum number of colors $k$ so that every $k$-coloring of $G$ is distinguishing. We generalize this concept to edge-coloring by defining an alternative index $\theta'(G)$. We consider $\theta'$ for some families of graphs and find its connection with edge-cycles of the automorphism group. Then we show that $\theta'(G)=2$ if and only if $G\simeq K_{1,2}$ and $\theta'(G)=3$ if and only if $G\simeq P_4, K_{1,3}\textnormal{ or } K_3$. Moreover, we prove some auxiliary results for graphs whose distinguishing threshold is 3 and show that although there are infinitely many such graphs, but they are not line graphs. Finally, we compute $\theta'(G)$ when $G$ is the Cartesian product of simple prime graphs.
		
		\noindent\textbf{Keywords:} {distinguishing coloring, distinguishing threshold, edge-distinguishing threshold}
		
		\noindent\textbf{Mathematics Subject Classification}:  05C09, 05C15, 05C25, 05C30
	\end{abstract}

	\section{Introduction}\label{intro}
	Symmetries are not always desirable, as sometimes it can make situations confusing. Breaking graphs' symmetries via coloring is an old and thriving subject in graph theory which is started in 1977, when Babai considered infinite graphs whose symmetries can be broken using only two colors~\cite{Babai1977}. The concept was redefined by Albertson and Collins~\cite{albertson1996symmetry} in 1996 as \emph{distinguishing coloring} of a graph $G$ which is a vertex coloring that is only preserved by the identity automorphism. The \emph{distinguishing number} $D(G)$ is the minimum number of colors require for such a coloring~\cite{albertson1996symmetry}.
	
	There are many results about the distinguishing number, see for example~\cite{Collins2006} by Collins and Trenk and~\cite{klavzar2006} by Klav\v{z}ar, Wong and Zhu who independently showed that $\Delta+1$ colors are enough to distinguishingly color a connected graph, or~\cite{Imrich2007Infinite} by Imrich, Klav\v{z}ar and Trofimov who considered breaking symmetries of infinite graphs, or~\cite{Bogstad2004} by Bogstad and Cowen and~\cite{Imrich2006CartPower} by Imrich and Klav\v{z}ar who considered distinguishing the Cartesian products. Some bounds were given by Alikhani and Soltani~\cite{Alikhani2018lexico} for the distinguishing number of the lexicographic product of two connected graphs and the corona product was also studied by them for their distinguishing indices~\cite{Alikhani2017corona}.
	
	Furthermore, within the last two decades some generalizations to this type of coloring were generated. For example, Collins and Trenk~\cite{Collins2006} considered distinguishing colorings that are also proper and Kalinowski, Pil\'sniak and Wo\'zniak~\cite{Kalinowski2016Total} considered distinguishing total colorings. Moreover, Imrich et al.~\cite{Imrich2014Endo} generalized the concept by considering breaking graphs' endomorphisms instead of automorphisms, Ellingham and Schroeder~\cite{Ellingham2011} broke symmetries of graphs via partitioning and Laflamme, Nguyen Van Th\'{e} and Sauer~\cite{Laflamme2010} defined the distinguishing number for digraphs and posets.
	
	Among these generalizations, distinguishing graphs via edge colorings, which first studied by Kalinowski and Pil\'sniak~\cite{Kalinowski2015Edge}, is important to understand this paper. They defined the \emph{distinguishing index}, $D'(G)$, to be the minimum number of colors require to color edges of $G$ such that it is not preserve by non-identity automorphisms of $G$. Subsequently, they showed that for a connected graph we have $D'(G)\leq \Delta$ when $\vert G\vert \geq 6$ and $D'(G)\leq D(G)+1$ when $\vert G\vert \geq 3$, both of which were generalized to infinite graphs, see~\cite{broere2015} by Broere and Pil\'sniak and~\cite{Imrich2017Bounds} by Imrich et al.
	
	Two vertex colorings $c_1$ and $c_2$ of a graph $G$ are said to be \emph{equivalent} if there is an automorphism $\alpha\in\aut(G)$ such that for each vertex $v\in E(G)$ we have $c_1 (v)=c_2 (\alpha (v))$. Ahmadi, Alinaghipour and Shekarriz in~\cite{ahmadi2020number} defined several indices to count non-equivalent distinguishing vertex coloring for connected graphs. More specifically, they defined $\Phi_k (G)$ as the number of non-equivalent distinguishing coloring of the graph $G$ with $\{1,\ldots,k\}$ as the set of colors and $\varphi_k (G)$ as the number of non-equivalent $k$-distinguishing coloring of the graph $G$ with $\{1,\ldots,k\}$ as the set of colors. They showed that these indices are related as follows:
	\begin{equation}
		\Phi_k(G) = \sum_{i=D(G)}^k {k\choose i} \varphi_i(G).
	\end{equation}
	As an application, they used $\Phi_k (G)$ to exactly calculate $D(G\circ H)$, the distinguishing number of the lexicographic product of two connected graph $G$ and $H$. Moreover, Shekarriz et al.~\cite{Shekarriz2021-vsum} used this index to calculate the distinguishing number of some graph operations such as the vertex-sum, the corona product and the smooth rooted product. Anyhow, the indices $\Phi_k (G)$ and $\varphi_k (G)$ showed to be not easy ones to calculate in general, as to date explicit formulae are available only for paths, cycles, complete graphs and complete bipartaite graphs in~\cite{ahmadi2020number} and just recently for grids (the Cartesian product of paths) by Alikhani and Shekarriz~\cite{Shekarriz2021Cartesian}. 
	
	However, sometimes $\varphi_k (G)$ is very easy to calculate, eg. when we know that every $k$-coloring of $G$ has to be distinguishing. Ahmadi, Alinaghipour and Shekarriz~\cite{ahmadi2020number} defined another index, $\theta(G)$, as the least number of color $k$ such that every $k$-coloring is distinguishing. Then, it can be easily shown that whenever $k\geq\theta(G)$ we have
	\begin{equation}
		\varphi_k (G)=\frac{k! \stirling{n}{k}}{|\aut(G)|},
	\end{equation}
	where $\stirling{n}{k}$ denotes the Stirling number of the second kind. Obviously, $$D(G)\leq \theta(G)\leq |V(G)|.$$ Moreover, they also showed that $\theta(K_n)=\theta(\overline{K_n})=n$, $\theta(K_{m,n})=m+n$, $\theta(P_n)=\lceil\frac{n}{2}\rceil+1$, for $n\geq 2$, $\theta(C_n)=\lfloor\frac{n}{2}\rfloor+2$, for $n\geq 3$, and for $n\geq 5$, they have proved that $\theta(K(n,2))=\frac{1}{2}(n^2-3n+6)$ where $K(n,k)$ is the Kneser graph~\cite{ahmadi2020number}. 
	
	Shekarriz et al.~\cite{Shekarriz2021-theta} continued studying the distinguishing threshold by showing that $\theta(G)$ is related to the cycle structure of the automorphism group of $G$ and calculating the distinguishing threshold for all graphs in the Johnson scheme. The distinguishing threshold for the Cartesian product of connected graphs has been studied by Alikhani and Shekarriz~\cite{Shekarriz2021Cartesian}, who showed that when $G=G_1\square G_2 \square \ldots \square G_k$ is a prime factorization to mutually non-isomorphic connected graphs, we have 
	\[
	\theta (G)=\max \left\{ \left(\theta(G_i)-1\right)\cdot \frac{\vert G\vert}{\vert G_i \vert} \; : \; i=1,\dots,k \right\}+1.
	\]
	Furthermore, they showed that when $G$ is a connected prime graph and $k\geq 2$ is a positive integer, we have 
	\[
	\theta (G^k)=\vert G\vert^{k-1}\cdot\max\left\{ \frac{\vert G\vert +1}{2}, \left(\theta(G)-1\right)\right\}+1,
	\]
	where $G^k$ denotes the $k$th Cartesian power of $G$.
	
	Some other graph products were considered by Shekarriz et al.~\cite{Shekarriz2021-vsum} for their distinguishing thresholds. They studied $\theta(G)$ when $G$ is the vertex-sum of some graphs or $G$ is the smooth rooted product, the corona product or the lexicographic product of two graphs.
	
	This is almost evident that for a graph $G$, the indices $\Phi_k (G)$, $\varphi_k (G)$ and $\theta (G)$ can have their alternatives for edge coloring. Here in this paper, we extend the notion of distinguishing threshold to edge coloring and leave the other two for future considerations. The \emph{edge-distinguishing threshold} of $G$, denoted by $\theta' (G)$ is the least integer $k$ such that any edge coloring of $G$ with $k$ colors is distinguishing. Evidently, if symmetries of $G$ can be broken by edge coloring, we have $\theta'(G)\leq \vert E(G) \vert$.
	
	The (edge-)distinguishing threshold must be considered important in the context of breaking symmetries of graphs because it shows the minimum $k$ such that every $k$-(edge-) coloring is distinguishing. For less colors, if there are some such colorings, we have to check whether a coloring is actually distinguishing or not. But when there is no limitations on the number of colors, we can choose $k$ bigger than or equal to the (edge-)distinguishing threshold and be sure that our arbitrary (edge-)coloring is distinguishing.
	
	Our graphs here are finite, simple and connected unless otherwise stated. Undefined terms and notations might be found in~\cite{Chartrand2016graphs-and-digraphs} by Chartrand, Lesniak and Zhang.
	
	\section{Preliminaries}
	
	In this section we remind some important results about the distinguishing threshold. Afterwards, we present a brief description of a Lemma by Alikhani and Shekarriz~\cite[Lemma 3.1]{Shekarriz2021Cartesian}, which we need in Section~\ref{sec:Cartesian}. They used it only for vertex coloring but it also works for edge and total colorings. To state the lemma, we need to recall some definitions. Moreover, we remind some results about line graphs.
	
	\subsection{The distinguishing threshold}
	
	Graphs automorphisms can be represented by product of cyclic permutations. For an automorphism $\alpha$ and a vertex $v\in V(G)$, the ordered tuple  $\sigma=(v, \alpha(v), \alpha^2 (v),\ldots,\alpha^{r-1} (v))$ is a cycle of length $r$ if $r$ is the least integer such that $\alpha^{r}(v)=v$. The number of cycles of an automorphism $\alpha$ is shown by $\vert \alpha \vert$~\cite{Shekarriz2021-theta}. The distinguishing threshold is highly connected with this number.
	
	\begin{lemma}\label{lem:theta-max-cycle}\textnormal{\cite{Shekarriz2021-theta}}
		For any graph $G$, we have 
		\[\theta(G)=\max\left\{|\alpha|\;:\; \alpha\in \aut(G)\setminus\{\mathrm{id}\} \right\}+1.\]
	\end{lemma}
	
	Using this lemma, one can show that $\theta(G)=2$ if and only if $\vert G\vert=2$ and $\theta(G)=3$ only if $\vert G \vert=3$ or $G$ is a graph with some certain conditions of order $2p$ for a prime number $p$. More explicitly, since we need it in Section~\ref{sec:edt}, we have the following result.

	\begin{theorem}\label{theta3}\textnormal{\cite{Shekarriz2021-theta}}
		Let $G$ be a graph on $n$ vertices for which we have $\theta(G)=3$. Then, either
		\begin{itemize}
			\item[(a)] $n=3$, or
			\item[(b)] $n=2p$ where $p\neq 3,5$ is a prime number, $G$ is a connected bi-regular graph with the degree-partitions $\mathcal{A}_1 =\{v_1,\ldots,v_p\}$ and $\mathcal{A}_2 =\{u_1,\ldots,u_p \}$, and the induced subgraphs $G[\mathcal{A}_1]$ and $G[\mathcal{A}_2]$ are non-isomorphic circulant graphs.
		\end{itemize}
	\end{theorem}
	
	Let $\vert G\vert =n$. Then $\theta(G)=D(G)$ if and only if $G$ is either asymmetric, the complete graph $K_n$ or the empty graph $\overline{K_n}$~\cite{Shekarriz2021-theta}.
	
	\subsection{Automorphisms of the Cartesian product of graphs}

	Automorphisms of the Cartesian products were studied by Imrich (and independently by Miller) as follows. 
	
	\begin{theorem} \textnormal{\cite{HIK2011}} \label{autG}
		Suppose $\psi$ is an automorphism of a connected graph $$G = G_1 \Box G_2 \Box \dots \Box G_k$$ which decomposed into its prime factors. Then for any automorphism $\psi\in\aut (G)$ there is a permutation $\pi$ of the set $\{1, 2, \dots , k\}$ and there are isomorphisms $\psi_i \colon G_{\pi(i)} \longrightarrow G_i$, $i=1, \dots, k$, such that
		$$\psi(x_1, x_2, \dots, x_k) = (\psi_1 (x_{\pi(1)}), \psi_2 (x_{\pi(2)}), \dots, \psi_r (x_{\pi(k)})).$$
	\end{theorem}

	For a vertex $v = (v_1 , v_2 , \ldots , v_k )$ of the product graph $G = G_1 \square G_2 \square \ldots \square G_k$, the \emph{$G_i$-layer through $v$} is the induced subgraph
	$$G_{i}^{v}=G \left[ \{ x \in V (G) \ \vert \  p_{j}(x) = v_j \text{ for }j\neq i\}\right],$$
	
	\noindent where $p_j$ is the projection mapping to the $j^\text{th}$-factor of $G$ \cite{HIK2011}.
	
	The \emph{$i^{\text{th}}$-quotient subgraph of $G$} is the graph
	$$Q_{i}=G \diagup G_{i} \simeq G_1 \square \ldots \square G_{i-1} \square G_{i+1} \square \ldots \square G_k.$$
	It is clear that $G \simeq G_i \square Q_i$ \cite{HIK2011}.

	\subsection{Holographic coloring}
	
	In this section, we briefly recall~\cite[Lemma 3.1]{Shekarriz2021Cartesian} (which originally in its incomplete form was from~\cite{ola2018} by Gorzkowska and Shekarriz) and its required definitions. It is stated in Lemma~\ref{AUL2} bellow and will be used in Section~\ref{sec:Cartesian}.  
	
	Suppose that $G=G_1 \square G_2 \square \ldots \square G_k$ for some $k\geq 2$ forms a prime decomposition, and $f$ is a total coloring for $G$ (if $f$ is a vertex or an edge coloring, by giving a color to all its edges or vertices we can easily change it to a total coloring). Fix an ordering on vertices of each factor graph so that $V(G_{i})=\{ 1_{i},\ldots, m_i \}$. For each $j=1_i,\ldots,m_i$, let $u_j = ( 1_{1},1_{2}, \ldots, 1_{i-1}, j, 1_{i+1},\ldots, 1_{k} )$.
	
	For $\alpha\in \aut(Q_{i})$ the map $\varphi_{\alpha}:Q_i^{u_j}\longrightarrow Q_i^{u_t}$ is defined, using $\alpha$ and our fixed ordering of $Q_i$, so that $\varphi_{\alpha}$ maps the vertex of $Q_i^{u_j}$ with the same ordering as $x\in Q_i$ onto the vertex of $Q_i^{u_t}$ with the same ordering as $\alpha(x)$. Obviously, $\varphi_\alpha$ is an isomorphism and is called a \emph{lifting of $\alpha$}. 
	
	\emph{The (total) coloring of $Q_i^{u_j}$ (induced by $f$)}, denoted by $\check{Q}_{i}^{u_j}$, is the graph $Q_{i}^{u_j}$ together with the (total) coloring induced by $f$. For $\alpha\in\Aut(Q_{i})$, the colored graph $\check{Q}_{i}^{u_j}$ is \emph{$\alpha$-equivalent} to $\check{Q}_{i}^{u_t}$ if there is a (total) color-preserving isomorphism $\varphi: Q_{i}^{u_j}\longrightarrow Q_{i}^{u_t}$ which is a lifting of $\alpha$ or $\alpha^{-1}$.
	
	Let $e=u_{i}v_i$ be an edge of $G_i$. Then $\overline{Q}_i^e$ is a vertex-colored graph isomorphic to $Q_i$ whose vertex set consists of edges of $G$ of the form $(u_i, x)(v_i, x)$ for some $x\in V(Q_{i})$. Two vertices $(u_i, x)(v_i, x)$ and $(u_i, y)(v_i, y)$ are adjacent in $\overline{Q}_i^e$ if $x$ is adjacent to $y$ in $Q_i$. Color each vertex $(u_i, x)(v_i, x)$ of $\overline{Q}_i^e$ by $f((u_i, x)(v_i, x))$ and denote the resulting vertex-coloring by $\hat{\overline{Q}}_i^e$. Similarly, the colored graph $\hat{\overline{Q}}_i^e$ is \emph{$\alpha$-equivalent} to $\hat{\overline{Q}}_i^{e'}$ if there is a vertex-color-preserving isomorphism $\vartheta:\overline{Q}_i^e \longrightarrow \overline{Q}_i^{e'}$ which is a lifting of $\alpha$ or $\alpha^{-1}$.
	
	For a (total) coloring $f$ of $G$, color each vertex $u_j$ by $\check{Q}_{i}^{u_j}$ and color each edge $e=u_j u_{\ell}$ by $\hat{\overline{Q}}_i^{e}$. This total coloring of $G_i$ is called \emph{holographic total coloring of $G_i$ induced by $f$} and is denoted $G_i^f$. Similarly, $G_i^f$ is \emph{equivalent to} $G_j^f$ if there is a total-color-preserving isomorphism from one to another. In this case we write $G_i^f \simeq G_j^f$.
	
	Alikhani and Shekarriz proved the following lemma and used it to calculate the distinguishing threshold of the Cartesian product graphs. However, it can also be used for edge and total colorings.

	\begin{lemma}\label{AUL2}\textnormal{\cite{Shekarriz2021Cartesian}} 
		Let $k\geq 2$ and $G=G_1 \square G_2 \square \ldots \square G_k$ be a connected graph decomposed into Cartesian prime factors. A (total) coloring $f$ is a distinguishing coloring for $G$ if and only if for each $i=1,\ldots, k$  we have
		
		\begin{itemize}
			\item[i.] $G_i^f \not\simeq G_j^f$ for all $j=1,\ldots,k$ such that $j\neq i$, and
			\item[ii.] for each $\alpha\in \aut(Q_i)$ and for each non-identity $\beta\in\Aut(G_i)$, there is a vertex $v\in G_i$ or an edge $e\in E(G_i)$ such that either $\check{Q}_i^v$ and $\check{Q}_i^{\beta(v)}$ or $\hat{\overline{Q}}_i^e$ and $\hat{\overline{Q}}_i^{\beta(e)}$ are not $\alpha$-equivalent.
		\end{itemize}
	\end{lemma}
	
	\subsection{Line graphs}
	Recall that the line graph of a graph $G$ is the graph $L(G)$ with $E(G)$ as its vertices, and two edges of $G$ are adjacent in $L(G)$ if and only if they are incident in $G$~\cite{Chartrand2016graphs-and-digraphs}. It is very well known result by Whitney~\cite{Whitney1932}, which (in modern terminology) can also be found in~\cite[Theorem 6.25]{Chartrand2016graphs-and-digraphs}, that for connected graphs $G_1$ and $G_2$ if $L(G_1)\simeq L(G_2)$ then $G_1 \simeq G_2$ unless $G_1 \simeq K_3$ and $G_2\simeq K_{1,3}$. It is easy to see that an edge coloring of $G$ is distinguishing if and only if the corresponding vertex coloring of $L(G)$ is distinguishing. Therefore, we have $D'(G)=D(L(G))$ and more specifically here, $\theta'(G)=\theta(L(G))$.
	
	In 1970 Beineke~\cite{BEINEKE1970} proved a theorem which characterizes all line graphs as graphs which forbids nine induced subgraphs. The most important of such subgraphs is $K_{1,3}$, i.~e., every line graph is claw-free.

	\section{Edge-distinguishing threshold}\label{sec:edt}

	Given an edge-coloring $c$, the \emph{palette} at a vertex $x$ is the set of colors of its incident edges.	Clearly, if every vertex has different palette from others, then $c$ is a distinguishing coloring~\cite{Kalinowski2015Edge}. Moreover, two edge colorings $c_1$ and $c_2$ of a graph $G$ are said to be \emph{equivalent} if there is an automorphism $\alpha\in\aut(G)$ such that for each edge $e=uv\in E(G)$ we have $c_1 (e)=c_2 (\alpha (e))$.
	
	Note that when $G$ is a connected simple graph and there is a non-identity automorphism of $G$ that fixes all the edges of $G$, we have $G\simeq K_2$. Furthermore, if $G$ is neither $K_2$ nor a disconnected graph with a connected component isomorphic to $K_2$ or having at least two isolated vertices, then $G$ has an edge-distinguishing coloring. 
	
	In the following theorem, we calculate the edge-distinguishing threshold for some well-known classes of graphs. We must only note that $\theta'(G)=1$ if and only if $G$ is asymmetric, and that $\theta'(G)\leq \vert E(G)\vert$ for all graphs whose symmetries can be broken by edge coloring.
	
	\begin{theorem}\label{TH:theta'-simple}
		For $m\geq 2$ and $n\geq 3$ we have
		\begin{itemize}
			\item[1.] $\theta' (P_n)=\left\lfloor \frac{n}{2}\right\rfloor +1$,
			\item[2.] $\theta' (C_n)=\left\lfloor \frac{n}{2}\right\rfloor +2$,
			\item[3.] $\theta' (K_3)= 3$ and $\theta' (K_n)=\frac{(n-1)(n-2)}{2}+2$ when $n\geq 4$,
			\item[4.] $\theta' (K_{m,n})=mn-m+1$, and
			\item[5.] $\theta' (K_{1,m})=m$.			
		\end{itemize}
	\end{theorem}
	\begin{proof}
		\begin{itemize}
			\item[1.] The only non-identity automorphism of $P_n$ is a reflection which induce $\left\lfloor \frac{n}{2}\right\rfloor$ orbits on the edges. Therefore, every edge coloring with $\left\lfloor \frac{n}{2}\right\rfloor+1$ colors has to be distinguishing. On the other hand, there is a non-distinguishing coloring with $\left\lfloor \frac{n}{2}\right\rfloor$ colors which uses all these colors on the first $\left\lfloor \frac{n}{2}\right\rfloor$ edges. Consequently, the result follows.
			
			\item[2.] Since $C_n$ is isomorphic to its own line graph $L(C_n)$, we have $\theta' (C_n)=\theta (C_n)$. Now, the result follows from Theorem 3.3 of~\cite{ahmadi2020number}.
			
			\item[3.]  It can be checked directly (or via item 2 because $K_3\simeq C_3$) that $\theta' (K_3)= 3$. Suppose that $n\geq 4$ and choose two vertices $u,v\in K_n$. Color edges of $K_n$ with $\frac{(n-1)(n-2)}{2}$ colors so that every edges of $K_n$ but those that are incident with $v$ receive a distinct color. For a vertex $x$ that is neither $u$ nor $v$, color the edge $vx$ the same color as what is given to $ux$. Use a new color on the edge $uv$. The resulting $\frac{(n-1)(n-2)}{2}+1$ coloring is not distinguishing because the transposition of $u$ and $v$ is a non-identity automorphism that keeps this coloring. 
			
			If we have used $k\geq\frac{(n-1)(n-2)}{2}+2$ colors on the edges of $K_n$, then by the pigeonhole principle, for each pair of vertices $u$ and $v$, the pallet of colors used on the edges incident to $u$ has a color that is not presented in the pallet of $v$. Hence, no vertex can be mapped onto another by a color preserving automorphism, and consequently this coloring has to be distinguishing.
			
			\item[4.] Suppose that $\mathcal{A}$ and $\mathcal{B}$ are the bipartition of $K_{m,n}$ of sizes $m$ and $n$ respectively. There is a non-identity automorphism of $K_{m,n}$ which transposes two vertices $v_1$ and $v_2$ of $\mathcal{B}$ and acts trivially on other vertices. Let $c$ be an edge coloring  with $mn-m$ colors which gives all the edges but those that are incident to $v_1$ a different color and for each $x\in\mathcal{A}$, we have $c(x v_1)=c(x v_2)$. Then $c$ is a non-distinguishing edge coloring. Therefore, we have $\theta'(K_{m,n})\geq mn-m+1$.
			
			On the other hand, if we color edges of $K_{m,n}$ by $mn-m+1$ colors, then by the pigeonhole principle and the fact that $n\geq 3$, the pallet of each vertex is different from others, so it cannot be mapped on another by an edge color preserving automorphism. Consequently, $\theta'(K_{m,n})\leq mn-m+1$.
			
			\item[5.] It is evident if two edges of $K_{1,m}$ have the same color, then there are an automorphism which transpose these two edges. Consequently, $\theta'(K_{1,m})>m-1$ and the result follows.
		\end{itemize}
	\end{proof}

	For simplicity, we make use of some notations. Suppose that $\alpha\in\Aut(G)$ and $e=uv\in E(G)$. Then, as a result of the faithful action of $\aut(G)$ on the edge set $E(G)$, we can think of $\alpha(e)$ as the edge $\alpha(u)\alpha(v)$. Using this action, every automorphism decomposes into cycles of edges. Similar to Shekarriz et al.~\cite{Shekarriz2021-theta}, we represent an automorphism of $G$  by a product of cyclic permutations of edges  where this representation is unique up to a permutation of cycles. So, the ordered tuple  $\sigma=(e, \alpha(e), \alpha^2 (e),\ldots,\alpha^{r-1} (e))$ forms a cycle of length $r$ provided that $r$ is the least integer such that $\alpha^{r}(e)=e$. The number of edge cycles of $\alpha$ is shown by $\vert \alpha \vert_e$. Then the following lemma is an alternative to Lemma~\ref{lem:theta-max-cycle}.
	
	\begin{lemma}\label{max-lem}
		Let $G$ be a graph whose symmetries can be broken by edge coloring. Then $$\theta' (G)=\max \left\{\vert \alpha \vert_e \; :\; \alpha\in\aut(G)\setminus \{\id\}\right\}+1.$$
	\end{lemma}
	\begin{proof}
		Let $t=\max\left\{|\alpha|_e \;:\; \alpha\in \aut(G)\setminus\{\mathrm{id}\} \right\}$ and $\sigma\in\aut (G)$ such that $\sigma\neq \mathrm{id}$ and $\vert \sigma\vert_e=t$. Color edges of $G$ using $t$ colors so that for $i=1,\ldots, t$, a single color $c_i$ is assigned to the edges of the $i$-th cycle of $\sigma$. This edge coloring is not distinguishing and so we have $\theta (G)\geq t+1$.
		
		Conversely,  let $\alpha$ be a non-identity automorphism of $G$. Because symmetries of $G$ can be broken by edge coloring, there is at least one cycle of $\alpha$ which has more than one edge. Then, since $\vert\alpha\vert_e\leq t$, by the pigeonhole principle any edge coloring of $G$ with $(t+1)$ colors uses at least 2 different colors on the edges of at least one of the cycles of $\alpha$. This means that $\alpha$ is not an edge-color-preserving automorphism. Because $\alpha$ is arbitrary non-identity automorphism, we have $\theta(G)\leq t+1$.
	\end{proof}
	
	Many results for the edge-distinguishing threshold can be retrieved from those of the distinguishing threshold using the line graph arguments. For example, one can deduce that $\theta'(G)=1$ or $\theta'(G)=\infty$ when $G$ is an infinite graph by using Theorem 2.11 of~\cite{Shekarriz2021-theta} and the fact that the line graph of an infinite graph is again an infinite graph. Moreover, by Whitney's theorem~(\cite[Theorem 6.25]{Chartrand2016graphs-and-digraphs}) and a result for the distinguishing threshold~\cite[Theorem 2.10]{Shekarriz2021-theta}, we have $\theta'(G)=D'(G)$ if and only if $G$ is asymmetric, $G\simeq K_{1,n}$ or $G\simeq K_3$. The following result can also be proven using a similar line graph argument, however we present another proof for it here.
	
	\begin{theorem}\label{theta'=2}
		Let $G$ be a connected graph whose edge-distinguishing threshold is 2. Then $G\simeq K_{1,2}$.
	\end{theorem}	
	\begin{proof}
		Suppose that $\alpha\in \aut(G)$ is a non-identity automorphism. Since $\theta'(G)=2$, then by Lemma \ref{max-lem} we have $\vert \alpha\vert_e =1$. Note that if $\deg (v)>1$ for a vertex $v\in V(G)$, then we must have $\alpha(v)=v$ because else we must have  $\vert \alpha\vert_e \geq 2$ which is not possible. Moreover, if there is an edge fixed by $\alpha$, then again we must have $\vert \alpha\vert_e \geq 2$, a contradiction to our assumption. If every vertex of $G$ has degree 1, then $G$ must have a connected component isomorphic to $K_2$, which cannot happen due to its violation with our implicit assumption that $G$ can be distinguished by edge coloring. Therefore, it can be deduced that the degree of only one vertex of $G$ is bigger than 1.
		
		Consequently, $G$ must be isomorphic to $K_{1,n}$ for some $n\geq 2$. By Theorem \ref{TH:theta'-simple} we know that $\theta'(K_{1,n})=2$ if and only if $n=2$. Hence, $G\simeq K_{1,2}$.
	\end{proof}
	
	Unlike the distinguishing threshold, Theorem~\ref{theta'=2} has an extension to a higher number. 
	We need the following lemma about the distinguishing threshold in to extend Theorem~\ref{theta'=2} to the case when the edge-distinguishing threshold is 3. Just note that $ \left[v, \mathcal{A}\right]$ means the set of edges whose one of their endpoints is on the vertex $v$ and  another belongs to the set $\mathcal{A}$.
	
	\begin{lemma}\label{lem:theta3-deg}
		Let $G$ be a graph such that $\theta(G)=3$ and $\vert G\vert>4$. Then for each $v\in \mathcal{A}_i$ we have $$3\leq \left| \left[v, \mathcal{A}_{i+1}\right]\right|\leq p-3,$$ where the addition in $\mathcal{A}_{i+1}$ is given modulo 2.
	\end{lemma}
	
	\begin{proof}
		Based on Theorem~\ref{theta3}, we have $\vert G\vert =2p$ where $p\neq 3,5$ is a prime number, $G$ is a bi-regular graph on $\mathcal{A}_1 =\{u_1,\ldots , u_p\}$ and $\mathcal{A}_2 =\{v_1,\ldots , v_p\}$ partition vertices of $G$ so that $G[\mathcal{A}_1]$ and $G[\mathcal{A}_2]$ are nonisomorphic circulant graphs.
		
		Suppose that $\sigma\in\aut(G)$ be a non-identity automorphism. Then by choosing one vertex from each partition, say $u\in\mathcal{A}_1$ and $v\in\mathcal{A}_2$, we have an ordering on $\mathcal{A}_1$ and $\mathcal{A}_2$ using $\sigma$ such that $u_0=u, u_1=\sigma(u), u_2=\sigma^2 (u),\ldots, u_{p-1}=\sigma^{p-1}(u)$ and $v_0=v, v_1=\sigma(v), v_2=\sigma^2 (v),\ldots, v_{p-1}=\sigma^{p-1}(v)$.
		
		If $u$ has no neighbor in $\mathcal{A}_2$, then neither do other vertices in $\mathcal{A}_1$, and hence $G$ is disconnected. Since the dihedral group $D_{2p}$ is a subgroup of $\aut(G[\mathcal{A}_1])$ (see~\cite[Remark 2.5]{Shekarriz2021-theta}) and $p\geq 7$, it can be inferred that there is a reflection automorphism on $G[\mathcal{A}_1]$ which has more than 3 cycles. Now, by Lemma~\ref{lem:theta-max-cycle}, we must have $\theta(G)>4$, which is a contradiction.
		
		So suppose that $u$ has exactly one neighbor in $\mathcal{A}_2$. For $x\in\mathcal{A}_1$, denote its only neighbor in $\mathcal{A}_2$ by $f(x)$, and for $v\in\mathcal{A}_2$ denote its only neighbor in $\mathcal{A}_1$ by $b(v)$. Then, again because $D_{2p}$ is a subgroup of $\aut(G[\mathcal{A}_1])$ and $p\geq 7$, there is a reflection automorphism $\alpha\in\aut(G[\mathcal{A}_1])$ which fixes $u$. Therefore, there is automorphism $\beta\in\aut(G)$ such that
		\[\beta(y)=\left\{ \begin{gathered}
			\alpha(y)\hspace{24mm} y\in\mathcal{A}_1 \\
			f\left(\alpha(b(y))\right) \hspace{13mm} y\in\mathcal{A}_2
		\end{gathered}\right. .\]
		It is clear that $\beta$ has at least $2\cdot\lceil\frac{p}{2}\rceil=p+1\geq 8$ vertex-cycles and consequently by Lemma~\ref{lem:theta-max-cycle} we must have $\theta(G)\geq 9$, a contradiction.
		
		Now, suppose that $u$ has exactly two neighbors in $\mathcal{A}_2$, say $f_1 (u)$ and $f_2 (u)$. 
		Since $p$ is an odd prime greater than 5, we can find a reflection automorphism $\gamma_1\in\aut(G[\mathcal{A}_1])$ such that it fixes $u$. Put $v_i=f_1(u)$ and $v_j=f_2(u)$. Without loss of generality, suppose that $j>i$. Then one of the numbers $j-i$ and $p+i-j$ is odd and the other is even. let $k$ be the middle number of the odd one and find the reflection automorphism $\gamma_2\in\aut(G[\mathcal{A}_2])$ that fixes $v_k$ and map $v_i$ and $v_j$ onto each other. Then $\gamma_1\cup\gamma_2$ is an automorphism of $G$ which has at least $2\cdot \lceil\frac{p}{2}\rceil = p+1$ cycles. Therefore, by Lemma~\ref{lem:theta-max-cycle} we must have $\theta(G)\geq p+2\geq 9$ which is a contradiction.
		
		Using nonadjacencies instead of adjacencies and with a similar argument, we can easily prove that $u$ cannot be adjacent to all, all but one and all but two vertices of $\mathcal{A}_2$. Now, the statement follows. 
	\end{proof}

	When $\theta(G)=3$ and $u\in\mathcal{A}_i$, $i=1,2$, the set of neighbors of $u$ outside $\mathcal{A}_i$ is shown by $\mathcal{N}_{out}(u)$. Moreover, a double-star, $DS_n$ is a graph on $n+2$ vertices $1,\ldots,n,n+1,n+2$ and $n+1$ and $n+2$ are adjacent to vertices $1,\ldots,n$. Equipped with these notations and the line graph arguments, we prove the following result which is a little surprising because there are infinitely many graphs whose distinguishing thresholds are 3.
	
	\begin{theorem}\label{theta'=3}
		Let $G$  be a connected graph whose edge-distinguishing threshold is 3. Then $G$ is either isomorphic to $P_4$, $K_{1,3}$ or $K_3$.
	\end{theorem}
	
	\begin{proof}
		Suppose that $G$ is a graph such that $\theta'(G)=3$. Then we must have $\theta(L(G))=3$. Suppose that $H=L(G)$. By Theorem~\ref{theta3} either $\vert H\vert=3$ or $\vert H\vert= 2p$ for a prime number $p\neq 3,5$. If $\vert H \vert=3$, then from the list of all graph on 3 vertices, only $K_3$ and $P_3$ are line graphs. Therefore, $G$ must be either $K_3$ or $K_{1,3}$, the only graphs whose line graphs is $K_3$, or it is $P_4$ whose line graph is $P_3$.
		
		We claim that if a vertex $v\in G$ has degree greater than 2, then $v$ is fixed by all automorphisms of $G$. Suppose that $\alpha\in \aut(G)$ and $v\in G$ such that $\deg_G (v)\geq 3$. If $\alpha$ moves $v$, i.~e., $\alpha(v)\neq v$, then it means that $\alpha$ induces at least 3 cycles on edges of $G$. Therefore, by Lemma~\ref{max-lem} we must have $\theta'(G)\geq 4$, a contradiction. Consequently, $\alpha(v)=v$.
		
		Suppose that $e=uv$ be an edge of $G$ and $\vert H \vert >3$. Since $\theta (H)=3$, Lemma~\ref{lem:theta3-deg} implies that the vertex $e\in V(H)$ has at least 3 neighbors in $\mathcal{N}_{out}(e)$. This translate into $G$ so that $e$ is incident on at least 3 other edges. By the pigeonhole principle, vertex degree of either $u$ or $v$ is greater than 2. Therefore by the the claim above, every edge of $G$ is incident on a fixed vertex. Moreover, $G$ can possess at most 2 fixed vertices because otherwise we have $\vert \sigma \vert_e \geq 3$ for every $\sigma\in\aut (G)$, which violates Lemma~\ref{max-lem}. Other vertices are then of degree 1 or 2.
		
		Suppose that $x$ and $y$ are vertices of $G$ of degree greater than 2. They cannot be adjacent because otherwise the edge $e=\{xy\}$ has to be fixed by all automorphism of $G$, which means that $H$ has to have a fixed vertex. This is wrong by Theorem~\ref{theta3}. Other vertices of $G$ cannot all be of degree 1 because else $G$ must be disconnected. Moreover, since every non-identity automorphism of $G$ induces exactly 2 cycles on $E(G)$, it can be deduced that all vertices other than $x$ and $y$ are of degree 2. Combining this with the fact that every edge of $G$ is incident on a fixed vertex, it is deduced that $G$ is a double star $DS_p$ for $p\geq 7$. The double star has no fixed vertices and besides, it has an automorphism transposing two vertices of degree 2 and fixes other edges, which means that $\theta'(G)\geq 2p-4\geq 10$ by Lemma~\ref{max-lem}. Both of these assertions contradict our assumptions.
		
		If $G$ has only one fixed vertex, namely $x$, then all other vertices are of degree 1 because every edge of $G$ is incident on a fixed vertex. Therefore, every edge is incident on $x$ and therefore $G\simeq K_{1,2p}$ for $p\geq 7$. By Theorem~\ref{TH:theta'-simple}, this implies that $\theta'(G)=2p\geq 14$ which contradicts our very assumption that $\theta'(G)=3$.
		
		Therefore, there is no graph other than  $K_3$, $K_{1,3}$ and $P_4$ whose edge-distinguishing threshold is 3.
	\end{proof}
	
	The consequence of the previous theorem over the distinguishing threshold is given by the following.
	
	\begin{corollary}
		Let $G$ be a graph on more than 3 vertices and $\theta(G)=3$. Then $G$ is not a line graph.
	\end{corollary}
	
	There are many cases that the line graph argument cannot help or it is easier to consider the edge coloring directly. For example suppose that $G$ is the Cartesian product of two arbitrary graphs $H_1$ and $H_2$. Then although it is true that $\theta'(G)=\theta(L(G))$, calculating $\theta (L(G))$ can be sometimes more complicated than directly calculating $\theta'(G)$. Therefore, it is worth calculating the edge-distinguishing threshold using direct arguments of edge coloring.

	\section{Total distinguishing threshold}
	
	To imply Lemma~\ref{AUL2} for edge coloring, the resulting holographic coloring is a total one. Therefore, one might ask about the \emph{total-distinguishing threshold} $\theta'' (G)$ which is the minimum integer $k$ such that any $k$-total coloring of $G$ is distinguishing. A complete answer to this question requires another full-length paper, but we can define some alternative definitions here such as total cycles of an automorphism $\alpha$, denoted by $\vert \alpha\vert_t$, and easily show that $\vert \alpha\vert_t=\vert \alpha \vert +\vert \alpha \vert_e$. 
	
	By the pigeonhole principle, any total coloring with $\theta(G)+\theta'(G)-1$ has to induce a distinguishing coloring on vertices or edges of $G$. Consequently, we have $\theta''(G)\leq\theta(G)+\theta'(G)-1$. Additionally, for any graph $G$ we have $\theta''(G)\leq \vert V(G) \vert +\vert E(G)\vert-1$. This bound is sharp because it is attained by the star $K_{1,n}$ for $n\geq 2$.

	\section{Edge-distinguishing threshold for the Cartesian products}\label{sec:Cartesian}
	As noted in the introduction, using Lemma~\ref{AUL2} the authors in~\cite{Shekarriz2021Cartesian} calculated the (vertex) distinguishing threshold for the Cartesian product graphs. However, they stated Lemma~\ref{AUL2} so that it enables us to calculate edge-distinguishing threshold of the Cartesian products.

	\begin{theorem}\label{theta'-cart}
		Let $G=G_1\square G_2 \square \ldots \square G_k$ for $k\geq 2$ be a prime factorization to mutually non-isomorphic connected graphs. Then 
		\begin{equation*}
			\theta'(G)=  \max \left\{ \vert \beta \vert_e\cdot \vert V(Q_i)\vert+\vert \beta \vert\cdot\vert E(Q_i)\vert :  i=1,\dots,k, \beta\in\aut(G_i)\setminus\{\id \} \right\}+1.
		\end{equation*}
	\end{theorem}
	
	\begin{proof}
		Let $1\leq i\leq k$ and 
		\begin{equation}\label{q}
			q(\beta)=b\cdot \vert V(Q_i)\vert+a\cdot\vert E(Q_i)\vert
		\end{equation}
		for some $\beta\in\aut(G_i)\setminus\{\id \}$ such that $a=\vert \beta\vert$ and $b=\vert\beta\vert_e$. Color edges of $G$ using $q(\beta)$ colors as follows. Let $f$ be a non-distinguishing total coloring with $\vert \beta \vert_t =a +b$ colors and suppose that colors $1,\ldots,a$ are used only on vertices and colors $a+1,\ldots,a+b$ are used only on edges of $G_i$. Since $\vert \beta\vert_t=a+b$, there is such a coloring.
		
		Fix one ordering of vertices and another ordering of edges of $Q_i$. For $u\in V(G_i)$ and every vertex $v\in V(G_i)$ from the cycle of $\beta$ containing $u$, color edges of $Q_i^v$  using $\vert E(Q_i)\vert$ colors. And for $e=\{xy\}\in E(G_i)$ and every edge $d$ in the edge-cycle of $\beta$ containing $e$, color vertices of $\overline{Q}_i^d$ (which are actually edges of $G$ joining vertices of quotient graphs such as $Q_i^x$ and $Q_i^y$) with $\vert V(Q_i)\vert$ colors. This coloring is $\id$-equivalent and therefore by Lemma~\ref{AUL2} ($ii$) we have $f$ is not distinguishing. Consequently we have $\theta'(G)\geq \max\{q(\beta):\beta\in\aut(G_i)\setminus \{\id\}, i=1,\ldots,k\}+1$.
		
		Since for $i\neq j$ we know $G_i\not\simeq G_j$, for any edge coloring $c$ of $G$ it is obvious that $G_i^c \not\simeq G_j^c$. Let 
		\begin{equation}\label{r}
			r= \max\{q(\beta):\beta\in\aut(G_i)\setminus \{\id \}, i=1,\ldots,k\}.
		\end{equation}
		To show that $r+1$ is the edge-distinguishing threshold, it is enough to show that item ($ii$) of Lemma~\ref{AUL2} is true for an arbitrary edge coloring of $G$ using $r+1$ colors.
		
		Now, suppose that $c$ is an edge coloring of $G$ using $r+1$ colors. Let $\alpha\in \Aut(Q_i)$, $\beta\in \Aut(G_i)\setminus \{\id \}$, $a=\vert \beta\vert$ and $b=\vert \beta \vert_e$. For each $i\in\{1,\ldots,k\}$, by pigeonhole principle either there are at least one $u\in V(G_i)$  such that $\check{Q}_i^{u}$ and $\check{Q}_i^{\beta (u)}$ are not $\alpha$-equivalent, or there is at least one $e\in E(G_i)$ such that $\hat{\overline{Q}}_i^{e}$ and $\hat{\overline{Q}}_i^{\beta (e)}$ are not $\alpha$-equivalent. This is because the relevant quotient graphs have different colors from each other. Therefore, item ($ii$) of Lemma~\ref{AUL2} is also met, and consequently, $c$ is a distinguishing coloring.
	\end{proof}
	
	It must be noted that in some cases (eg. when $G$ is the Cartesian product of paths and cycles or other well-formed simple graphs), $a$ and $b$ in the previous theorem would become $\theta(G_i)-1$ and $\theta'(G_i)-1$ respectively. However in general, there are graphs whose automorphism that has the greatest number of vertex cycles does not have the greatest number of edge cycles. As a result, we have to check for which automorphism $\beta\in\aut(G_i)$ the number $q(\beta)$ of Equation~\ref{q} generates the maximum, and then take the maximum for all factor graphs.
	
	Theorem \ref{autG} implies that $\Aut(G^k)\cong \mathrm{Sym}(k)\oplus \Aut(G)^k$. We use this to prove the following.
	
	\begin{theorem}\label{theta'-power}
		Let $G$ be a connected prime graph, $k\geq 2$ be a positive integer and $r$ is defined in Equation~\ref{r}. Then $$\theta' (G^k)=\max\left\{ \frac{k}{2}\cdot\vert G\vert^{k-1}\cdot \vert E(G)\vert, r \right\}+1.$$
	\end{theorem}
	
	\begin{proof}
		First note that $\vert E(G^k)\vert= k\cdot \vert G \vert^{k-1} \cdot \vert E(G) \vert$, and $G^k$ has an automorphism that transposes two factors which has exactly $\frac{1}{2}\vert E(G^k)\vert$ edge cycles. so we must have $$\theta' (G^k)\geq\max\left\{ \frac{k}{2}\cdot\vert G\vert^{k-1}\cdot \vert E(G)\vert, r \right\}+1.$$
		
		Now, let $c$ be an arbitrary edge coloring for $G^k$ with $\max\left\{ \frac{k}{2}\cdot\vert G\vert^{k-1}\cdot \vert E(G)\vert, r \right\}+1$ colors. From proof of Theorem \ref{theta'-cart} we know that for each $i=1,\ldots,k$ item ($ii$) of Lemma~\ref{AUL2} holds because the number of colors are more than $r+1$. If for some $i,j\in \{1,\ldots,k\}$, $i\neq j$ we have $G_i^c\simeq G_j^c$, then we must have an isomorphism that maps $G_i^c$ onto $G_j^c$. Whenever this isomorphism maps a vertex $v_1\in G_i$ onto $v_2\in G_j$ and an edge $e_1\in E(G_i)$ onto $e_2\in E(G_j)$, it maps $\check{Q}_i^{v_1}$ onto $\check{Q}_j^{v_2}$ and $\hat{\overline{Q}}_i^{e_1}$ onto $\hat{\overline{Q}}_j^{e_2}$, which in turn mean that the edge colorings of $Q_i^{v_1}$ and $Q_j^{v_2}$ is equivalent, and the vertex colorings of $\overline{Q}_i^{e_1}$ and $\overline{Q}_j^{e_2}$ are also equivalent. 
		
		Therefore, $G_i^c\simeq G_j^c$ implies that there must be an edge-color-preserving automorphism $\gamma$ of $G^k$ such that $\vert \gamma \vert_e = \frac{k}{2}\cdot\vert G\vert^{k-1}\cdot \vert E(G)\vert$. This is not possible because number of colors is at least one more than this. Consequently, item ($i$) of Lemma~\ref{AUL2} is also met and $c$ is a distinguishing edge coloring.
	\end{proof}

	\section*{Acknowledgment}
	This paper is dedicated to professor Wilfried Imrich for his 80th birthday. We honorably pay homage for his lifelong contribution to many different areas of graph theory, including very much to product graphs and breaking symmetries via coloring. 
	
	The second author was supported by Yazd University’s support number 99/50/910. The authors declare no conflict of interest.

	\bibliographystyle{plain}
	\bibliography{bibliography}
	
\end{document}